\documentclass[11pt,reqno]{amsart}
\usepackage[T1]{fontenc}
\usepackage{graphicx}
\usepackage{setspace}
\usepackage{empheq}

\usepackage{color}
\definecolor{MyLinkColor}{rgb}{0,0,0.4}

\newtheorem{thm}{Theorem}[section]
\newtheorem{prop}[thm]{Proposition}

\theoremstyle{remark} 
\newtheorem{rem}[thm]{Remark}

\numberwithin{equation}{section}   

\title[Recovery of Stokes waves from velocity measurements ]{Recovery of Stokes waves from velocity measurements on an axis of symmetry}

\author[B.--V. Matioc]{Bogdan--Vasile Matioc}
\address{Institut f{\"u}r Angewandte Mathematik, Leibniz Universit{\"a}t Hannover, Welfengarten~1, 30167 Hannover, Deutschland.}
\email{matioc@ifam.uni-hannover.de}


\usepackage[colorlinks=true,linkcolor=MyLinkColor,citecolor=MyLinkColor]{hyperref} 

\begin{document}

\begin{abstract}
We provide a new method to recover the profile  of  Stokes waves, and more generally of waves with smooth vorticity,  from  measurements  of the horizontal velocity component 
on a vertical axis of symmetry of the wave surface.
Although we consider  periodic waves only,   the extension   to  solitary waves   is straightforward.
\end{abstract}

\maketitle

\section{Introduction}\label{S:1}
A Stokes wave is a two-dimensional periodic wave that travels at constant speed over a flat bed solely under the influence of gravity.
The flow beneath a Stokes wave is irrotational and the  surface profile is symmetric with respect to both crest and trough lines.
Initiated by Stokes \cite{S1880}, the rigorous study of Stokes waves was developed steadily to uncover many fascinating aspects of these remarkable water flows, cf.  e.g. \cite{BT03,  Con11, Con13, AC11,  To96}.  

The problem of reconstructing the free surface of regular Stokes waves from measured flow data, such as the pressure on the fluid bed, is of great importance in the field of fluid dynamics, cf. \cite{BL95, CC13, De11, Do90, KC94, OD12, Tsai05}.
A very simple relation  between the pressure on the bed and the elevation of the water surface is obtained when assuming that the pressure is hydrostatic, see the discussion in \cite{CC13, OD12}.
This method is used for tsunami detection, but, as it  does not account for nonlinear  wave effects, prediction errors   are quite frequent \cite{Do90}.
On the other hand, the transfer function approach   \cite{ES08, KC10} uses a linear formula relating the  wave surface elevation to the pressure on the fluid bed, a large body of research being dedicated to determining a suitable transfer function.
This latter method restricts primarily to   linear irrotational  waves.
Therefore, it  does not  account for nonlinear effects and sometimes  overestimates the   height of large waves \cite{Tsai05}.

Within the exact theory of irrotational water waves, a nonlinear and nonlocal relation  between the wave profile of solitary waves and the pressure on the bed was derived in \cite{De11, OD12}. 
An alternative exact local recovery formula  was obtained in \cite{C12, Hsu14}.
For a recovery formula in the setting of  steady periodic waves we refer to \cite{Cl13, CC13}.
In fact it is theoretically possible to recover the  profile of solitary waves from the pressure on the bed even when allowing for rotational flows, provided that the vorticity distribution is real-analytic \cite{DH14An}.

The goal of the present note is to present a new exact method to recover the profile of periodic traveling waves and the velocity field within the flow from measurements of the horizontal velocity on a vertical axis of symmetry of the wave.
Our analysis is   within the setting of rotational periodic waves with a smooth vorticity distribution, the extension to solitary waves is straightforward. 
In particular, our method can be used to reconstruct the free surface of Stokes waves.
That the profile of rotational periodic gravity waves with a smooth vorticity is determined by the  horizontal velocity on a vertical axis of symmetry -- which is  either a crest or a trough line -- is quite intriguing and it
underlines the strong connection between the flow  and the free  surface   motion. 
This intricate connection is  suggested also by the result of \cite{BM14} where it is shown that the symmetry of the free wave surface  can be characterized  intrinsically in terms of the underlying flow. 

The main tool used in our approach is a very recent observation pertaining to the regularity of rotational flows: any weak solution of the water wave problem 
possesses real-analytic streamlines, the radius of analyticity being independent of the position 
in the fluid, cf. \cite{EM14}.
Finally, we emphasize that the horizontal velocity on crest  or trough lines has been measured in many laboratory  experiments \cite{GCHJ03, HHY12, YS70, SEG92, SCJ01}
and therefore we expect that the practical accuracy of our recovery method can be verified numerically.

\section{Equivalent formulations of the water wave problem}\label{S:2}
We consider  steady periodic water waves propagating at the surface of a two-dimensional fluid body with  wave speed $c>0$.
The Cartesian coordinates $(x,y,z)$ are fixed such that the
$x$-axis is the direction of wave propagation, the $y$-axis points vertically upwards, and the flow is independent of the  $z$-coordinate.
The fluid layer is bounded from below by the flat bed $y=-d$, whereby $d>0$. 
Neglecting viscosity effects and temperature variations, the motion of the fluid is governed by the Euler equations.
As we consider  steady waves only, the space-time dependence of the free surface, of the pressure, and of the velocity field has the form $(x - ct)$. 
Eliminating time by the change of coordinates $(x - ct, y,z) \mapsto (x, y,z),$ in the new coordinates system
in which the origin moves in the direction of propagation of the wave with wave speed $c,$
the wave is stationary and the flow is steady, the equations of motion being encompassed  in the following system
\begin{equation}\label{eq:2}
\left\{
\begin{array}{rllll}
(u-c) u_x+vu_y&=&-P_x&\text{\qquad in $\Omega_\eta$},\\
(u-c) v_x+vv_y&=&-P_y-g&\text{\qquad in $\Omega_\eta$},\\
u_x+v_y&=&0&\text{\qquad in $\Omega_\eta$},\\
P&=&P_0&\text{\qquad on $y=\eta(x)$},\\
v&=&(u-c)\eta'&\text{\qquad on $y=\eta(x)$},\\
v&=&0&\text{\qquad on $ y=-d$}.
\end{array}
\right.
\end{equation}
Hereby, $y=\eta(x)$ is the surface of the wave, $\Omega_\eta:=\{(x,y)\,:\,-d<y<\eta(x)\}$ is the fluid domain, 
$P$ denotes the dynamic pressure, $u$ is the horizontal velocity, $v$ is the vertical velocity,
$g$ is the gravity of Earth, the constant $P_0$ is the uniform air pressure, and the water's density is set to be $\rho=1.$

Classical solutions $(u,v,P,\eta)$ of \eqref{eq:2} belong to the following regularity class
 \begin{equation}\label{eq:r}
(u,v,P,\eta)\in \big(C^{1}_{per}(\overline\Omega_\eta)\big)^3\times C^{2}_{per}(\mathbb{R}),
 \end{equation}
the subscript $per$ being used  to express the $\lambda-$periodicity of $(u,v,P,\eta)$ with respect to the variable $x$, and where the constant $\lambda>0$ is the wavelength. 
A last assumption   we make is that the wave speed $c$ exceeds the horizontal velocity of each particle within the fluid
 \begin{equation}\label{UC}
 \sup_{\Omega_\eta}u<c.
\end{equation}

The  formulation \eqref{eq:2} can be also recast in terms of the streamfunction $\psi$
 which is defined by 
\begin{equation}
 \label{Psi}
\begin{array}{rllll}
\psi(x,y):=-p_0+\displaystyle\int_{-d}^y(u(x,s)-c)\, ds
\end{array}
\qquad\text{for $(x,y)\in\overline\Omega_\eta$},
\end{equation}
the constant  $p_0<0$, which represents the relative mass flux, being given by
\begin{equation}\label{MF}
p_0:=\displaystyle\int_{-d}^{\eta(x)}(u(x,s)-c)\, ds.
\end{equation}
It follows from the third and the last two equations in \eqref{eq:2} that  $p_0$  is independent of $x$.
Let us also observe that  since $\nabla\psi=(-v,u-c)$, the vorticity of the flow is given by the scalar function 
$\omega:=u_y-v_x=\Delta\psi.$

Because of \eqref{UC},  the map $\mathcal{H}:\overline\Omega_\eta\to\overline\Omega$, with $\Omega:=\mathbb{R}\times(p_0,0),$ defined 
by the relation
\begin{equation}\label{HH}
 \mathcal{H}(x,y):=(q,p)(x,y)=(x,-\psi(x,y))\qquad\text{for $(x,y)\in\overline\Omega_\eta$} 
\end{equation}
 is a diffeomorphism, that is $\mathcal{H}\in\mbox{\rm Diff}\,^2(\Omega_\eta,\Omega)$.
 Moreover, it follows from \eqref{eq:2},  cf. e.g. \cite[Section 2]{BM11},  that $\omega\circ \mathcal{H}^{-1}=:\gamma$ depends only on the variable $p$.
The function $\gamma\in C([p_0,0])$ is the so-called   vorticity function.
Note that both $\omega$ and  $\gamma$ are zero constant  functions for Stokes flows.  
Using Bernoulli's law, we find that $(\eta,\psi)\in C^2_{per}(\mathbb{R})\times C^2_{per}(\overline\Omega_\eta)$ solves the following free boundary problem 
\begin{equation}\label{eq:Psi}
\left\{
\begin{array}{rllll}
\Delta\psi&=&\gamma(-\psi)&\text{\qquad in $\Omega_\eta$},\\
|\nabla\psi|^2+2g(y+d)&=&Q&\text{\qquad on $y=\eta(x)$},\\
\psi&=&0&\text{\qquad on $y=\eta(x)$},\\
\psi&=&-p_0&\text{\qquad on $ y=-d$},
\end{array}
\right.
\end{equation}
with the positive constant $Q:=2(E-P_0).$  
In terms of $\psi$ the condition \eqref{UC} appears as  
 \begin{equation}\label{UC1}
 \sup_{\Omega_\eta}\psi_y<0.
\end{equation}

For the existence theory, but also when studying the properties of solutions of  \eqref{eq:Psi}, cf. \cite{CoSt04, EM14}, it is 
useful to express the hydrodynamical problem in terms of the height function  $h:\overline\Omega\to\mathbb{R}$ that is defined by 
$h:=y\circ \mathcal{H}^{-1}+d.$ 
The function $h\in C^2_{per}(\overline\Omega)$ associates to each pair $(q,p)\in\overline\Omega$, the height of the water particle $(x,y):=\mathcal{H}^{-1}(q,p)$ above the bed.
Particularly, we have that $\eta(x)=h(x,0)-d$ for all $x\in\mathbb{R}.$ 
 It follows easily from the definition of $h$ and \eqref{eq:Psi}, cf. e.g.  \cite{CoSt04}, that $h$ solves  a quasilinear elliptic equation with nonlinear boundary conditions 
\begin{equation}\label{eq:hod}
\left\{
\begin{array}{rllll}
(1+h_q^2)h_{pp}-2h_ph_qh_{pq}+h_p^2h_{qq}-\gamma h_p^3&=&0&\text{\qquad in $\Omega$},\\[1ex]
1+h_q^2+(2gh-Q)h_p^2&=&0&\text{\qquad on $p=0$},\\
h&=&0&\text{\qquad on $p=p_0$},
\end{array}
\right.
\end{equation}
while the assumption \eqref{UC} is now equivalent to
 \begin{equation}\label{UC2}
     \inf_\Omega h_p>0.
\end{equation}

In the next Section, we  use the equivalence of the three formulations within the set of classical solutions.
The proof of this result follows easily from the arguments  presented in  \cite[Chapter 3]{Con11} and in \cite{BM11}.
\begin{prop}[Equivalence of the three formulations]\label{P1}
 The following are equivalent:
\begin{enumerate}
 \item[$(i)$] the equations \eqref{eq:2} and \eqref{UC} with $(u,v,P,\eta)$ satisfying \eqref{eq:r};
 \item[$(ii)$] the equations \eqref{eq:Psi} and \eqref{UC1} with $(\eta,\psi)\in C^2_{per}(\mathbb{R})\times C^2_{per}(\overline\Omega_\eta)$ and $\gamma\in C([p_0,0])$;
 \item[$(iii)$] the equations \eqref{eq:hod} and \eqref{UC2} with $h\in C^2_{per}(\overline\Omega) $ and $\gamma\in C([p_0,0]).$
\end{enumerate}
\end{prop}

\section{The main result and discussion}\label{S:3}
In order to state our main result, we emphasize first that any (vertical) axis of symmetry of the wave must either be a trough or a crest line.
This is a simple consequence of the  real-analyticity of the wave profile, a property which holds a priori  for solutions in the class \eqref{eq:r}, cf. \cite[Corollary 1.2]{EM14}. 
The  main result of this paper is the following theorem.

\begin{thm}\label{MT1} Consider a regular
\footnote{By regular wave we mean a  solution $(u,v,P,\eta)$ of the hydrodynamical problem \eqref{eq:2} that satisfies \eqref{eq:r} and \eqref{UC}. 
Notice that this class of solutions ensures that the three formulations are equivalent, cf. Proposition \ref{P1}. 
The condition \eqref{UC} is satisfied by regular Stokes waves automatically,
only the Stokes wave of greatest height having the property that $u=c$ at the crest where the free water surface forms an angle of 120 degrees, cf. \cite{To96}.} 
 two-dimensional symmetric and periodic gravity wave with  vorticity $\omega\in C^\infty(\overline\Omega_\eta)$ that travels over the flat bed $y=-d$ with wave speed $c>0.$
Furthermore, assume that   the horizontal component $u$  of the  velocity field is known on the   line of symmetry
$$\{(x_0,y)\,:\, -d\leq y\leq \eta(x_0)\}$$ 
of the wave.
Then, we can recover the wave profile, the velocity field and the pressure distribution within the flow. 
\end{thm}
Before proving Theorem \ref{MT1} it is worthwhile to add some remarks.

\begin{rem} 
\begin{itemize}
 \item[$(a)$]  Theorem \ref{MT1} is true in particular for Stokes waves for which the vorticity is zero. 
We emphasize that a constant nonzero vorticity, for which  the assumption $\omega\in C^\infty(\overline\Omega_\eta)$  is also trivially satisfied, is relevant in  many  physical situations, cf. the discussion in  \cite{SP88}.
 \item[$(b)$]
We remark that to the best of our knowledge, there is no rigorous proof for the existence of two-dimensional periodic steady gravity waves which are not symmetric.
For criteria which ensure the symmetry of such waves with respect to both trough and crest lines we refer to \cite{CoEhWa07, CoEs04_b, BM14, OS01}.
\item[$(c)$] In the proof of Theorem \ref{MT1}, we do  not use of the assumption of periodicity.
Particularly, our result shows that also the surface of solitary waves can be recovered when knowing the velocity field on a axis of symmetry.
\item[$(d)$] The values of $u$ on the vertical symmetry line $[x=x_0]$ and the wave speed are not both needed, but just the values of   $u-c$ on this axis of symmetry. 
\item[$(e)$] The proof of Theorem \ref{MT1} uses to a large extent the fact that the streamlines of the flow are all real-analytic, the radius  of analyticity being uniform in the fluid.
An  open question is whether the radius of analyticity can be chosen to be  $\infty.$
\end{itemize}

\end{rem}

\begin{proof}[Proof of Theorem \ref{MT1}]
Let $\eta,$ $(u,v)$, and $P$ be the unknown wave profile, velocity field, and pressure distribution, respectively.
Because of the equivalence result in Proposition \ref{P1}, we only need to determine the height function $h$. 
Observing that the problem \eqref{eq:2}  is invariant with respect to horizontal translations we may set $x_0=0.$  
Hence, $\eta$ is symmetric with respect to the vertical line $[x=0].$

Our assumptions together with \eqref{MF} enable us to determine the relative mass flux constant $p_0$, cf. \eqref{MF}.
Knowing the strip $\Omega$ where the height function $h\in C^2_{per}(\overline\Omega)$  is defined, we next show that $h$ is well-determined  by the restriction of $(u-c)$  on the axis of symmetry $[x=0]$.
To this end, we recall that $\Delta\psi=\omega$  in $\Omega_\eta$ and that $\psi$ is constant on $\partial\Omega_\eta$.
From   our assumption $\omega\in C^\infty(\overline\Omega_\eta)$ and  the real-analyticity of the wave profile we deduce  that  $\psi\in C^\infty(\overline\Omega_\eta)$,  cf. \cite[Theorem 6.19]{GT01}.
Therewith, we get $\gamma=\omega\circ \mathcal{H}^{-1}\in C^\infty([p_0,0])$ and  also $h\in C^{\infty}_{per}(\overline\Omega)$.

Using the weak  elliptic maximum principle in the context of \eqref{eq:hod} as in \cite[Lemma 3.2]{BM14} shows that the symmetry of the free water surface with respect to $[x=0]$, that is $\eta(x)=\eta(-x)$ for all $x\in\mathbb{R},$ 
also implies  the symmetry of $h$, that is $h(q,p)=h(-q,p)$ for all $(q,p)\in\overline\Omega.$ 
Even more holds: there exists a constant $L>0$ such that  
\begin{align}\label{L:1}
 \|\partial_q^mh\|_{C^2(\overline\Omega)}\leq L^{m-1}(m-2)!\qquad\text{for all $m\geq 2,$}
\end{align}
cf. \cite[Theorem 3.3]{CLW13} (see also \cite{EM14} for a more general result).
The estimate \eqref{L:1} ensures the real-analyticity of the streamlines as   $[q\mapsto h(q,p)],$ $p\in[p_0,0],$ is  real-analytic.
Indeed,   \eqref{L:1} implies    for each  $(q_0,p)\in\overline\Omega$  that
 \begin{align}\label{BF}
 \Big |h(q,p)-\sum_{k=0}^n \frac{\partial_q^k h(q_0,p)}{k!}(q-q_0)^k \Big |
&  \leq \frac{\|\partial_q^{n+1}h\|_0}{(n+1)!}|q-q_0|^{n+1} 
\leq  
 \frac{(L|q-q_0|)^{n+1}}{L}\to_{n\to\infty} 0
\end{align}
if  $L|q-q_0|<1.$
In particular, the radius of convergence  of the Taylor series is independent of the point $(q_0,p).$
This latter property enables us to extend the function $h$ to the whole strip $\overline\Omega$, provided that we know it on the subset $[-(2L)^{-1}\leq q\leq (2L)^{-1}]\subset\overline\Omega,$
where   the Taylor series of $[q\mapsto h(q,p)]$ converges for each $p\in[p_0,0]$.

With this, our task reduces to determining $h$ in the rectangle ${[-(2L)^{-1}\leq q\leq (2L)^{-1}]}.$
We note that the symmetry of $h$  implies in  particular that $\partial_q^{2n+1}h(0,p)=0$ for all $n\in\mathbb{N}$ and $p\in[p_0,0]$.
This property and \eqref{BF} yield that
 \begin{align}\label{BF1}
 h(q,p)=\sum_{n=0}^\infty a_{2n}(p)q^{2n} \qquad\text{for $|q|<(2L)^{-1}$ and $p\in[p_0,0],$}
\end{align}
whereby $a_{2n}:=((2n)!)^{-1}\partial_q^{2n}h(0,\cdot)\in C^\infty([p_0,0])$ for all $n\in\mathbb{N}.$
 Similarly to \eqref{BF}, one can deduce from \eqref{L:1}  that
 \begin{equation}\label{BF2}
 \begin{aligned}
 &h_p(q,p)=\sum_{n=0}^\infty a_{2n}'(p)q^{2n},\qquad  h_q(q,p)=\sum_{n=1}^\infty 2na_{2n}(p)q^{2n-1}, \qquad h_{pp}(q,p)=\sum_{n=0}^\infty a_{2n}''(p)q^{2n},\\
 &h_{qp}(q,p)=\sum_{n=1}^\infty 2na_{2n}'(p)q^{2n-1},\qquad h_{qq}(q,p)=\sum_{n=1}^\infty 2n(2n-1)a_{2n}(p)q^{2n-2}
\end{aligned}
\end{equation}
for $|q-q_0|\leq (2L)^{-1}$ and  all $p\in[p_0,0].$

Since   the values of $(u-c)$ on the  symmetry line
$[x=0]$ are known, we can determine the coefficient function $a_0$ as follows.
From the definition of $h$ we obtain that $h(0,-\psi(0,y))=y+d$ for all $y\in[-d,\eta(0)],$
whereby, in virtue of \eqref{Psi} and \eqref{MF},
\begin{equation*} 
\psi(0,y)=-p_0+\displaystyle\int_{-d}^{y}(u(0,s)-c)\, ds,\qquad\text{ $y\in[-d,\eta(0)].$}
\end{equation*}
Because of \eqref{UC1}, the map $-\psi(0,\cdot):[-d,\eta(0)]\to[p_0,0]$ is bijective and its inverse is thus also known.
With the notation from Section \ref{S:2} it is the function  
$y\circ \mathcal{H}^{-1}(0,\cdot):[p_0,0]\to [-d,\eta(0)]$.
Hence, $a_0$ is determined by the formula $a_0=y\circ \mathcal{H}^{-1}(0,\cdot)+d:[p_0,0]\to [0,d+\eta(0)].$

To finish the proof, we show that $a_0$ determines all the other coefficient functions $a_{2n},$ $ n\in\mathbb{N}, $ $n\geq 1.$
Indeed, plugging the Taylor expansions \eqref{BF} and \eqref{BF1} into the first equation of \eqref{eq:hod} yields, after identifying for a fixed $p\in[p_0,0] $ the coefficients of each power of $q$, that 
\begin{align}
&A_0+C_0+D_0=0 \qquad\text{and}\qquad  A_{2n}+B_{2n}+C_{2n}+D_{2n}=0\qquad\text{for $n\geq1.$}\label{R2}
\end{align}
Hereby, we used the following relations
\begin{align*}
 &(1+h_q^2)h_{pp}=\sum_{n=0}^\infty A_{2n}q^{2n},\quad -2h_qh_ph_{pq}=\sum_{n=1}^\infty B_{2n},\quad q^{2n}h_p^2h_{qq}=\sum_{n=0}^\infty C_{2n}q^{2n},\quad-\gamma h_p^3=\sum_{n=0}^\infty D_{2n}q^{2n}
 \end{align*}
 with
\begin{align*}   
&A_{2n}=a_{2n}''+\sum_{k=1}^{n}a_{2(n-k)}''\Big(\sum_{l=1}^k4l(k-l+1)a_{2l}a_{2(k-l+1)}\Big),\\
 & B_{2n}=-2\sum_{k=0}^{n-1}a_{2(n-k-1)}'\Big(\sum_{l=1}^{k+1}4l(k-l+2)a_{2l}a_{2(k-l+2)}'\Big),\\
 &C_{2n}=\sum_{k=0}^{n}(2k+1)(2k+2)a_{2k+2}\Big(\sum_{l=0}^{n-k}a_{2l}'a_{2(n-k-l)}'\Big),\\
 &D_{2n}=-\gamma\sum_{k=0}^{n}a_{2k}'\Big(\sum_{l=0}^{n-k}a_{2l}'a_{2(n-k-l)}'\Big).
\end{align*}
The first equation of \eqref{R2} is equivalent to
\begin{equation}\label{F1}
 a_2=\frac{\gamma a_0'}{2}-\frac{a_0''}{2(a_0')^2},
\end{equation}
whereby, in view of $u<c$ in $\overline\Omega_\eta$, we have $a_0'>0.$
Moreover, from the second equation of  \eqref{R2}  we find that 
\begin{align}
 a_{2n+2}=&\frac{1}{(2n+1)(2n+2)(a_0')^2}\left[-a_{2n}''-\sum_{k=1}^{n}a_{2(n-k)}''\Big(\sum_{l=1}^k4l(k-l+1)a_{2l}a_{2(k-l+1)}\Big)\right.\nonumber\\
 &+2\sum_{k=0}^{n-1}a_{2(n-k-1)}'\Big(\sum_{l=1}^{k+1}4l(k-l+2)a_{2l}a_{2(k-l+2)}'\Big)\nonumber\\
 &\left.-\sum_{k=0}^{n-1}(2k+1)(2k+2)a_{2k+2}\Big(\sum_{l=0}^{n-k}a_{2l}'a_{2(n-k-l)}'\Big)+\gamma\sum_{k=0}^{n}a_{2k}'\Big(\sum_{l=0}^{n-k}a_{2l}'a_{2(n-k-l)}'\Big)\right]\label{F2}
\end{align}
for all integers $n\geq 1$.
The relations \eqref{F1} and \eqref{F2} show that $a_{2n+2}$ is uniquely determined  by $a_0,\ldots, a_0^{(2n)}$,  $\gamma,\ldots,\gamma^{(2n)}$. 
This  completes the proof.
\end{proof}

\vspace{0.5cm}
\noindent{\bf Acknowledgements} 
The author  thanks the anonymous referees for the comments and suggestions which have improved the quality of the article.

\end{document}